\documentclass[12pt,reqno]{amsart}

\usepackage{amssymb}

\newcommand{\F}{{\mathbb F}}
\newcommand{\N}{{\mathbb N}}

\newcommand{\seq}{\subseteq}

\newcommand{\iot}{\iota}

\newtheorem{lemma}{Lemma}
\newtheorem{theorem}{Theorem}
\newtheorem{corollary}{Corollary}

\newcommand{\refl}[1]{\ref{l:#1}}
\newcommand{\reft}[1]{\ref{t:#1}}

\newcommand{\refs}[1]{\ref{s:#1}}
\newcommand{\refb}[1]{\cite{b:#1}}
\newcommand{\refe}[1]{\eqref{e:#1}}

\author{Sergei V. Konyagin}
\email{konyagin@ok.ru}
\address{Steklov Mathematical Institute, 8 Gubkina St, Moscow 119991, Russia}
\thanks{The first author was supported by grants 08-01-00208 from the Russian
  Foundation for Basic Research and NSh-3233.2008.1 from the Program
  Supporting Leading Scientific Schools.}

\author{Vsevolod F. Lev}
\email{seva@math.haifa.ac.il}
\address{Department of Mathematics, The University of Haifa at Oranim,
         Tivon 36006, Israel}

\title[The Erd\H os-Tur\'an problem]%
  {The Erd\H os-Tur\'an problem \\ in infinite groups}

\dedicatory{Dedicated to Mel Nathanson on the occasion of his 60th birthday}

\begin{document}
\baselineskip 16pt

\begin{abstract}
Let $G$ be an infinite abelian group with $|2G|=|G|$. We show that if $G$ is
not the direct sum of a group of exponent $3$ and the group of order $2$,
then $G$ possesses a perfect additive basis; that is, there is a subset
$S\subseteq G$ such that every element of $G$ is uniquely representable as a
sum of two elements of $S$. Moreover, if $G$ \emph{is} the direct sum of a
group of exponent $3$ and the group of order $2$, then it does not have a
perfect additive basis; however, in this case there is a subset $S\subseteq
G$ such that every element of $G$ has at most two representations (distinct
under permuting the summands) as a sum of two elements of $S$. This solves
completely the Erd\H os-Tur\'an problem for infinite groups.

It is also shown that if $G$ is an abelian group of exponent $2$, then there
is a subset $S\subseteq G$ such that every element of $G$ has a
representation as a sum of two elements of $S$, and the number of
representations of non-zero elements is bounded by an absolute constant.
\end{abstract}

\maketitle

\section{The background}\label{s:bckg}

A subset of an abelian semigroup is called an \emph{additive basis of order
$2$}, or \emph{basis} for short, if every element of the semigroup is
representable as a sum of two elements of the subset. We say that a basis is
\emph{perfect} if every element is represented uniquely, up to the order of
summands. The representation function of a basis associates with each element
the number of its (ordered) representations as a sum of two elements from the
basis. If the semigroup can be embedded into an involution-free group, then
for a basis to be perfect it is necessary and sufficient that its
representation function is bounded by $2$.

A famous open conjecture of Erd\H os and Tur\'an \refb{et} is that every
basis of the semigroup $\N_0$ of non-negative integers has unbounded
representation function; that is, if $S\seq\N_0$ is a set such that each
non-negative integer is representable as a sum of two elements of $S$, then
there are integers with arbitrarily many representations.

Most investigations related to the Erd\H os-Tur\'an conjecture study
representation functions of bases of $\N_0$ (see~\refb{ns} for a survey) or
consider the analogous problem for infinite abelian semigroups other than
$\N_0$, and also for infinite families of abelian semigroups (see~\refb{hh}).
In the present paper we are concerned with the latter line of research.

There are several noticeable cases where bases with bounded representation
functions are known to exist. As an example, Nathanson~\refb{n} proved that
the group of integers possesses a perfect basis. Ruzsa~\refb{r} showed that
if $p$ is a prime with $\Big(\frac{\;2\,}p\Big)=-1$, then the group
$\F_p\times\F_p$ possesses a basis such that every group element has at most
$18$ representations as a sum of two elements of this basis. (Here and below
for a prime $p$ we denote by $\F_p$ the finite field with $p$ elements. To
simplify the notation, we occasionally identify a field with its additive
group.) As a corollary, Ruzsa derived a result~\cite[Theorem~1]{b:r} which
easily implies that every finite cyclic group has a basis whose
representation function is bounded by an absolute constant, independent of
the order of the group. The approach of~\refb{r} was further developed by
Haddad and Helou~\refb{hh} to show that for any finite field $\F$ of odd
characteristic, the group $\F\times\F$ has a basis whose representation
function does not exceed $18$. In the case where $\F$ is a finite field of
characteristic $2$, a basis in $\F\times\F$ with a bounded representation
function was constructed in~\cite[Lemma 1]{b:gdt}, though the property we are
interested in has never been identified explicitly to our knowledge.

\section{The results}

For a subset $C$ of an abelian group and an integer $n\ge 1$ we write
  $$ nC := \{ nc \colon c \in C \}. $$
Our main result is
\begin{theorem}\label{t:inf}
Let $G$ be an infinite abelian group with $|2G|=|G|$.
\begin{itemize}
\item[(i)]  If $G$ is not the direct sum of a group of exponent $3$ and
    the group of order $2$, then $G$ has a perfect basis.
\item[(ii)] If $G$ is the direct sum of a group of exponent $3$ and the
    group of order $2$, then $G$ does not have a perfect basis, but has a
    basis such that every element of $G$ has at most two representations
    (distinct under permuting the summands) as a sum of two elements of
    the basis.
\end{itemize}
\end{theorem}

Clearly, if $G$ is an infinite abelian group with $|2G|<|G|$, then for any
basis $S$ of $G$ (and indeed, for any subset $S\seq G$ with $|S|=|G|$) there
is an element of $G$ having as many as $|G|$ representations of the form $2s$
with $s\in S$. In particular, this applies to infinite abelian groups of
exponent $2$. Similarly, for no infinite family of groups of exponent $2$ can
one find bases with uniformly bounded representation functions, even if the
groups of the family are finite. We show that, nevertheless, efficient bases
in such groups do exist if we exclude the zero element from consideration.
\begin{theorem}\label{t:exp2}
Each abelian group of exponent $2$ possesses a basis such that every non-zero
element of the group has at most $36$ representations as a sum of two
elements of this basis.
\end{theorem}

Combined with the result of Haddad and Helou and the corollary of Ruzsa's
result, mentioned in Section~\refs{bckg}, Theorems~\reft{inf} and~\reft{exp2}
readily yield
\begin{corollary}
Let $G$ be an abelian group. If $G$ is either infinite with $|2G|=|G|$, or
has prime exponent, then it possesses a basis with the representation
function bounded by an absolute constant (independent of the group), except
for the value of the function on the zero element in the case where $G$ is of
exponent $2$.
\end{corollary}

We notice that excluding the zero element for groups of exponent $2$ is
equivalent to disregarding representations with equal summands. To our
present knowledge, a universal constant $K$ may exist with the property that
each abelian group possesses a basis such that every element of the group has
at most $K$ representations as a sum of two distinct elements of this basis.

\section{The proofs}\label{s:proofs}

In this section we use the word ``basis'' both in the above-defined and
linear-algebraic meaning, adding the attribute \emph{linear} in the latter
case to avoid confusion.

Our argument depends on the axiom of choice, which we assume for the rest of
the paper.

To handle infinite groups of exponents $2$ and $3$, we need the following
lemma.
\begin{lemma}\label{l:FF}
If $G$ is an infinite abelian group of prime exponent $p$, then there exists
an algebraically closed field $\F$ of characteristic $p$ such that
$G\cong\F\times\F$.
\end{lemma}

The proof uses several facts, well known from algebra and set theory; namely,
\begin{itemize}
\item[(i)]   every vector space has a linear basis;
\item[(ii)]  an infinite set can be partitioned into two disjoint subsets
    of equal cardinality;
\item[(iii)] an infinite vector space over a finite field has the same
    cardinality as any of its linear bases;
\item[(iv)]  the field of rational functions over a finite field in the
    variables, indexed by the elements of an infinite set, has the same
    cardinality as this set;
\item[(v)]   the algebraic closure of an infinite field has the same
    cardinality as the field itself.
\end{itemize}
We notice that (i) follows easily from Zorn's lemma, while (ii)--(v) are not
difficult to derive from the basic set theory result saying that for any
infinite cardinal $m$, a union of at most $m$ sets, each of cardinality at
most $m$, has cardinality at most $m$.

\begin{proof}[Proof of Lemma \refl{FF}]
Considering $G$ as a vector space over the field $\F_p$, find a linear basis
$B$ of $G$ and fix a partition $B=B_1\cup B_2$, where $B_1$ and $B_2$ are
disjoint subsets of equal cardinality. For $i\in\{1,2\}$ denote by $G_i$ the
group of functions from $B_i$ to $\F_p$ with a finite support; thus,
$G_1\cong G_2$, and since $G$ is isomorphic to the group of functions from
$B$ to $\F_p$ with a finite support, we have $G\cong G_1\times G_2$. Let $\F$
be the algebraic closure of the field of rational functions over $\F_p$ in
the variables, indexed by the elements of $G_1$. By (iv) and (v), the
cardinality of $\F$ is equal to the cardinality of $G_1$. From (iii) we
conclude now that every linear basis of $\F$ has the same cardinality as
$B_1$ which, we recall, is a linear basis of $G_1$. Any bijection from $B_1$
to a linear basis of $\F$ determines a group isomorphism between $G_1$ and
the additive group of $\F$. As a result, we have $G_1\cong\F$, and hence also
$G_2\cong\F$, implying the assertion.
\end{proof}

For an abelian group $G$, an integer $n\ge 1$, and subsets $A,B,C\seq G$ we
write
\begin{align*}
  G_n    &:= \{g\in G\colon ng=0\}, \\
  A\pm B &:= \{ a\pm b \colon a\in A,\, b\in B \},
\intertext{and}
  A+B-C  &:= \{ a+b-c\colon a\in A,\, b\in B,\, c\in C \}.
\end{align*}
From $G/G_n\cong nG$ we conclude that if $|G|$ is infinite, then
$\max\{|G_n|,|nG|\}=|G|$.

Yet another result used in the proof of Theorem \reft{inf} is
\begin{lemma}\label{l:2G3G}
Let $G$ be an abelian group such that $2G$ is infinite. If $A,B\seq G$
satisfy
  $$ \max\{|A|,|B|\} < \min\{|2G|,|3G|\}, $$
then there exists an element $s\in G$ with $2s\notin A$ and
 $3s\notin B$.
\end{lemma}

\begin{proof}
Suppose for a contradiction that for every $s\in G$ we have either $2s\in A$,
or $3s\in B$. Without loss of generality we assume $B\seq 3G$, and we find
then a subset $U\seq G$ with $|U|=|B|$ and $B=\{3u\colon u\in U\}$.

Fix $w\in G$ with $3w\notin B$. For any $g\in G_3$ we have
 $3(w+g)=3w\notin B$, whence $2(w+g)\in A$ and therefore $2g\in-2w+A$. Now if
$s\in G$ satisfies $3s\in B$, then $s=u+g$ with some $u\in U$ and $g\in G_3$,
implying $2s=2u+2g\in 2u-2w+A\seq -2w+2U+A$. It follows that for every
 $s\in G$ we have either $2s\in -2w+2U+A$, or $2s\in A$; this, however, is
impossible as $|-2w+2U+A|\le|U||A|<|2G|$ and $|A|<|2G|$ by the assumptions.
\end{proof}

Eventually, we are ready to prove Theorem \reft{inf}.
\begin{proof}[Proof of Theorem \reft{inf}]
We split the proof into three parts.

\smallskip\noindent
{\bf 1.\,} First, suppose that $G$ is of exponent $3$. By Lemma~\refl{FF}, we
can assume $G=\F\times\F$, where $\F$ is an algebraically closed field of
characteristic $3$. Set
  $$ S := \{(x,x^2)\colon x\in\F\}. $$
For each pair $(u,v)\in G$, the number of representations of $(u,v)$ as a sum
of two elements of $S$ is the number of solutions of the equation
  $$ x^2 + (u-x)^2 = v;\quad x\in \F, $$
which is either $1$, or $2$. The assertion follows.

We remark that this argument above actually goes through for any odd prime
exponent; however, only the case of exponent $3$ is not covered by the proof
below.

\smallskip\noindent
{\bf 2.\,} Now suppose that $G=F\oplus\{0,h\}$, where $F$ is of exponent $3$
and $h$ has order $2$. As shown above, $F$ has a perfect basis $S$, and it is
immediate that $S\cup(h+S)$ is a basis of $G$ such that every element of $G$
has at most two representations (distinct under permuting the summands) as a
sum of two elements of this basis.

Assuming, on the other hand, that $G$ possesses a \emph{perfect} basis, we
write this basis as $T=T_0\cup(h+T_1)$ with $T_0,T_1\seq F$. Shifting $T$
appropriately, we assume furthermore that $0\in T_0$. To obtain a
contradiction we observe that the unique representation of $h$ as a sum of
two elements of $T$ has the form $h=t_0+(h+t_1)$ with $t_0\in T_0$ and
$t_1\in T_1$; hence, $2(t_1+h)=t_0+0$ gives two representations of $t_0$ as a
sum of two elements of $T$.

\smallskip\noindent
{\bf 3.\,} Turning to the general case, we denote by $\mu$ the initial
ordinal of the cardinal $|G|$ and consider a well-ordering
 $G=\{g_\iot\colon \iot<\mu\}$. Notice, that $\mu$ is a limit ordinal, and
hence the successor of any ordinal, smaller than $\mu$, is also smaller than
$\mu$.

We set $S_0:=\varnothing$ and construct a chain of subsets $S_\iot$, for each
ordinal $\iot\le\mu$, so that
\begin{itemize}
\item[--] $S_\iot\seq S_\rho$ whenever $\iot<\rho\le\mu$;
\item[--] if $\iot$ is a finite ordinal, then $S_\iot$ is finite, and if
    $\iot\le\mu$ is infinite, then $|S_\iot|\le|\iot|$;
\item[--] $g_\iot\in S_\rho+S_\rho$ whenever $\iot<\rho\le\mu$;
\item[--] for any ordinal $\iot\le\mu$ and element $g\in G$ there is at
    most one representation of $g$ as a sum of two elements of $S_\iot$.
\end{itemize}
The proof is then completed by observing that $S_\mu$ is a perfect basis of
$G$; hence, it suffices to show that the subsets $S_\iot$ can be constructed.

We use transfinite recursion, assuming that $\nu\le\mu$ and that $S_\iot$ has
already been found for each ordinal $\iot<\nu$, and constructing $S_\nu$. If
$\nu$ is a limit ordinal, then we put $S_\nu:=\cup_{\iot<\nu} S_\iot$. If
$\nu$ is a successor ordinal and $g_{\nu-1}\in S_{\nu-1}+S_{\nu-1}$, then we
put $S_\nu:=S_{\nu-1}$. In the remaining case where $\nu$ is a successor
ordinal and $g_{\nu-1}\notin S_{\nu-1}+S_{\nu-1}$ we put
$S_\nu:=S_{\nu-1}\cup\{s,t\}$, where $s,t\in G$ with $s+t=g_{\nu-1}$ are
chosen to satisfy the following conditions:
\begin{itemize}
\item[(a)] if $|3G|<|G|$, then $s,t\notin S_{\nu-1}+3G$;
\item[(b)] $s,t\notin S_{\nu-1}+S_{\nu-1}-S_{\nu-1}$;
\item[(c)] $2s,2t\notin S_{\nu-1}+S_{\nu-1}$;
\item[(d)] $s-t\notin S_{\nu-1}-S_{\nu-1}$;
\item[(e)] $2s-t,2t-s\notin S_{\nu-1}$.
\end{itemize}
Condition (a) is of technical nature and its exact purpose will be clarified
in the following paragraph, while the last four conditions ensure that the
unique representation property of $S_{\nu-1}$ is inherited by $S_\nu$. Thus,
to complete the proof it suffices to show that $s$ and $t:=g_{\nu-1}-s$
satisfying (a)--(e) can be found.

To this end we first observe that if $S_{\nu-1}$ is infinite, then condition
(e) excludes at most $|S_{\nu-1}|\le|\nu-1|<|\mu|=|G|$ options for $3s$, and
similarly (a)--(d) together exclude fewer, than $|G|$ options for $2s$.
Clearly, this conclusion remains valid if $S_{\nu-1}$ is finite. Therefore,
in view of Lemma~\refl{2G3G} we can assume that $|3G|<|G|$. Consequently,
securing (a) at each step of the construction, we have ensured that all
elements of $S_{\nu-1}$ fall into distinct cosets of $3G$, and in particular
each of $g_{\nu-1}+S_{\nu-1}$ and $2g_{\nu-1}-2S_{\nu-1}$ contains at most
one element from $3G$. Since (e) can be re-written as
  $$ 3s \notin (g_{\nu-1}+S_{\nu-1}) \cup (2g_{\nu-1}-S_{\nu-1}), $$
if $|3G|\ge 3$, then there exists $g\in G$ such that every $s\in g+G_3$
satisfies (e). As remarked above, (a)--(d) reduce to forbidding fewer than
$|G|$ values for $2s$; that is, forbidding fewer than $|G|$ cosets of $G_2$
for $s$. Since $|g+G_3|=|G|$ in view of $|3G|<|G|$, and every $G_2$-coset
intersects $g+G_3$ by at most one element, there exists $s\in g+G_3$ with an
admissible value of $2s$, proving the assertion.

Suppose, therefore, that $|3G|<3$. If $|3G|=1$, then $G$ is of exponent $3$,
the case which has been addressed above. If $|3G|=2$, then the identity
$g=-2g+3g$ shows that $G=G_3+3G$, and the sum is direct as if
 $g\in G_3\cap 3G$, then $2g=0$ (as $g\in 3G$ and $|3G|=2$) and $3g=0$ (as
$g\in G_3$), implying $g=0$. Consequently, $G$ is the direct sum of a group
of exponent $3$ and the group of order $2$. This completes the proof.
\end{proof}

Finally, we prove Theorem \reft{exp2}. As indicated in the introduction, the
construction employed in the proof is adopted from \refb{gdt}, where it is
used (in the finite-dimensional case) to find small codes with covering
radius $2$.

\begin{proof}[Proof of Theorem \reft{exp2}]
In view of Lemma~\refl{FF}, it suffices to show that if the field $\F$ of
characteristic $2$ is either finite or algebraically closed, then the group
$\F\times\F$ has a basis with the representation function bounded by $18$.
Clearly, we can assume $|\F|>2$.

We fix $d_1,d_2,d_3\in\F^\times$ with $d_1+d_2+d_3=0$, write
  $$ S_i := \{(x,d_i/x)\colon x\in\F^\times\};\ i\in\{1,2,3\}, $$
and put $S=S_1\cup S_2\cup S_3$. For $(u,v)\in\F\times\F$ let $r(u,v)$ denote
the number of representations of $(u,v)$ as a sum of two elements of $S$, and
for $i,j\in\{1,2,3\}$ denote by $r_{ij}(u,v)$ the number of representations
of $(u,v)$ as a sum of an element of $S_i$ and an element of $S_j$. Since the
sets $S_1,S_2$, and $S_3$ are pairwise disjoint, we have
  $$ r(u,v) = \sum_{i,j=1}^3 r_{ij}(u,v) $$
and furthermore,
  $$  r_{ij}(u,v) = \big| \{ x \in \F \setminus \{0,u\}
                 \colon d_i/x+d_j/(x+u)=v \} \big|;\quad i,j\in\{1,2,3\}. $$
The equation $d_i/x+d_j/(x+u)=v$ can be re-written as
\begin{equation}\label{e:1}
  vx^2+(uv+d_i+d_j)x+d_iu = 0
\end{equation}
and since it has a non-zero coefficient unless $(u,v)\ne (0,0)$, we have
$r_{ij}(u,v)\le 2$, except if $u=v=0$. It follows that $r(u,v)\le 18$ and to
achieve our goal it suffices to show that for any $(u,v)\ne (0,0)$ there are
$i,j\in\{1,2,3\}$ with $r_{ij}(u,v)>0$. We consider three cases.

If $u=0$ and $v\ne 0$, then $r_{12}(u,v)$ is the number of solutions of
$d_1/x+d_2/x=v$, which is $1$.

If $u\ne 0$ and $v=0$, then $r_{12}(u,v)$ is the number of solutions of
$d_1/x=d_2/(x+u)$; this leads to a non-degenerate linear equation, the
solution of which is distinct from both $0$ and $u$.

Finally, suppose that $u\ne 0$ and $v\ne 0$. In this case for $j=i$ equation
\refe{1} takes the form
\begin{equation}\label{e:2}
  vx^2+uvx+d_iu=0,
\end{equation}
and for $r_{ii}(u,v)$ to be non-zero it is necessary and sufficient that
\refe{2} has a solution (which automatically is then distinct from $0$ and
$u$). If $\F$ is algebraically closed, then we are done; suppose, therefore,
that $\F$ is finite. Since \refe{2} can be re-written as
  $$ (x/u)^2+(x/u) = d_i/(uv), $$
it has a solution if and only if $d_i/(uv)$ belongs to the image of the
linear transformation $x\mapsto x+x^2$ of the field $\F$ considered as a
vector space over $\F_2$. The kernel of this transformation is a subspace of
dimension $1$; hence its image is a subspace of $\F$ of co-dimension $1$.
(This image actually is the set of all the elements of $\F$ with zero trace,
but we do not use this fact.) Consequently,
  $$ d_1/(uv) + d_2/(uv) + d_3/(uv) = 0 $$
implies that at least one of $d_1/(uv),\,d_2/(uv)$, and $d_3/(uv)$ is an
element of the image.
\end{proof}

\bigskip

\end{document}